\documentclass{amsart}

\usepackage{amssymb,latexsym,amsmath,epsfig,amsthm}
\usepackage{graphicx}
\usepackage{cancel}
\usepackage{hyperref}
\usepackage{amsfonts}
\usepackage{latexsym}
\usepackage[english]{babel}
\usepackage{graphicx}
\usepackage[T1]{fontenc}

\hypersetup{colorlinks=true}\theoremstyle{plain}
\newtheorem{theorem}{Theorem}
\newtheorem{corollary}{Corollary}

\theoremstyle{definition}

\numberwithin{equation}{section}

\begin{document}

\title{ Stability of  a new generalized reciprocal type functional equation}

\author{I. Sadani}
\address{Department of Mathematics, University of Mouloud Mammeri, Tizi-Ouzou 15000, Algeria}
\email{sadani.idir@yahoo.fr}

\date{}
\maketitle

\begin{abstract}
%%%%% ABSTRACT %%%%%
In this paper, we investigate the generalized Hyers-Ulam stability of the following  reciprocal type functional equation
\begin{equation*}f(2x+y)+f(2x-y)=\frac{2f(x)f(y)\displaystyle{\sum_{\substack{k=0\\  \text{$k$ is even}}}^{ l}2^{l-k}\binom{l}{k}f(x)^{\frac{k}{l}}f(y)^{\frac{l-k}{l}}}}{\left(4f(y)^{\frac{2}{l}}-f(x)^{\frac{2}{l}}\right)^l}\end{equation*}

in  non-zero real and non-Archimedean spaces. 
\end{abstract}

% Next you must introduce the contents of your article

\section{Introduction}
Since the birth of  theory of stability,  which was first proposed by Ulam \cite{ul}  in the form of a question and after the first answer due to Hyers \cite{hy},  many results  have been obtained for many functional equations in different spaces. In this work we are interested in the case of reciprocal functional equations.

%The study of the stability of functional equations was instigated by the famous question of Ulam [1] during a Mathematical Colloquium at the University of Wiskonsin in the year 1940. In the successive year, Hyers [2] provided a partial answer to the question of Ulam. Later, Hyers's result was extended and generalized for a Cauchy functional equation by Bourgin [3], Th.M. Rassias [4], Gruber [5], $\mathrm{A}\mathrm{o}1\langle \mathrm{i}[6], \mathrm{I}\cdot \mathrm{M}$. Rassias [7] and $\mathrm{G}\dot{\mathrm{a}}$vruta [$8]$ in var- ious adaptations. After that several stability articles, many textbooks and research mono- graphs have investigated the result for various functional equations, also for mappings with more general domains and ranges; for instance, see [9-16] and [17].

 In 2010, Ravi and Senthil Kumar \cite{ku} studied  the stability od the
 reciprocal type functional equation
\begin{equation}\label{reci1}f(x+y)=\frac{f(x)f(y)}{f(x)+f(y)},\end{equation}  

where $f:X\rightarrow \mathbb{R}$ is a mapping with $X$ as the space of non-zero real numbers.

% The reciprocal function $r(x)=\frac{c}{\chi}$ is a solution of the functional equation (1.1). The functional equation (1.1) holds good for the reciprocal formula' of any electric circuit with two resistors connected in parallel [19].
%%%%%%%%%%%%%%%%%%%%%%%%%%%%%%%%%%%%%%%%%%%%%%%%%%
%%%%%%%%%%%%%%%%%%%%%%%%%%%%%%%%%%%%%%%%%%%%%%%%%%%%%%
%%%%%%%%%%%%%%%%%%%%%%%%%%%%%%%%%%%%%%%%%%%%%%%%%%%%%%%
% Ravi et al. \cite{rab} obtained the solution of a new generalized reciprocal-type functional equation in two variables of the form
%$$
%r(x+y)= \frac{/\sigma r(x+(k-1)y)r((k-1)x+y)}{r(x+(k-1)y)+r((k-1)x+y)},$$  
%  where $k>2$ is a positive integer, and investigated its generalized Hyers-Ulam stability in non-Archimedean fields. Then Senthil Kumar et al. \cite{ku1} found a general solution of a reciprocal-type functional equation
%
%$$f(x+y)= \frac{f(\frac{k_{1}x+k_{2}y}{k})f(\frac{k_{2}x+k_{1}y}{k})}{f(\frac{k_{1}x+k_{2}y}{k})+f(\frac{k_{2}x+k_{1}y}{k})}$$  
%
%and investigated its generalized Hyers-Ulam-Rassias stability in non-Archimedean fields, where $k>2, k_{1}$ and $k_{2}$ are positive integers with $k=k_{1}+k_{2}$ and $k_{1}\neq k_{2}$.
%%%%%%%%%%%%%%%%%%%%%%%%%%%%%%%%%%%%%%%%%%%%%%%%%%%%%%%%%%%%%%%%%%%%%%%%%%
%%%%%%%%%%%%%%%%%%%%%%%%%%%%%%%%%%%%%%%%%%%%%%%%%%%%%%%%%%%%%%%%%%%%%%%%%%%%%

In 2014, Kim and Bodaghi \cite{kim} introduced and studied the  stability of the quadratic reciprocal functional equation

\begin{equation}\label{reci2}f(2x+y)+f(2x-y)= \frac{2f(x)f(y)[4f(y)+f(x)]}{(4f(y)-f(x))^{2}}\end{equation}
%%%%%%%%%%%%%%%%%%%%%%%%%%%%%%%%%%%%%%%%%%%%%%%%%%%%%%%%%%%%%%%%%%%%%%%%%%%%%
%Then the functional equation was generalized in \cite{eb} as
%
%$$f((a+1)x+ay)+f((a+1)x-ay)= \frac{2f(x)f(y)[(a+1)^{2}f(y)+a^{2}f(x)]}{((a+1)^{2}f(y)-a^{2}f(x))^{2}}$$
%where $a\in \mathrm{Z}$ with $a\neq 0, -1$. In \cite{eb}, the authors established the generalized Hyers-Ulam- Rassias stability for the functional equation (1.5) in non-Archimedean fields.
%%%%%%%%%%%%%%%%%%%%%%%%%%%%%%%%%%%%%%%%%%%%%%%%%%%%%%%%%%%%%%%%%%%%%%%%%%%%%%%%%%%%%%%%%%%%%%%%

In the same year, Ravi et al. \cite{ra14} investigated the generalized Hyers-Ulam-Rassias stability of a reciprocal-quadratic functional equation of the form

\begin{equation}\label{reci3}r(x+2y)+r(2x+y)= \frac{r(x)r(y)[5r(x)+5r(y)+8\sqrt{r(x)r(y)}]}{[2r(x)+2r(y)+5\sqrt{r(x)r(y)}]^{2}}   
\end{equation}
in intuitionistic fuzzy normed spaces. 

In 2017, Kim et al. \cite{kim1}  introduced and studied the stability of the reciprocal-cubic functional equation

\begin{equation}\label{reci4}c(2x+y)+c(x+2y)= \frac{9c(x)c(y)[c(x)+c(y)+2c(x)^{\frac{1}{3}}c(y)^{\frac{1}{3}}(c(x)^{\frac{1}{3}}+c(y)^{\frac{1}{3}})]}{[2c(x)^{\frac{2}{3}}+2c(y)^{\frac{2}{3}}+5c(x)^{\frac{1}{3}}c(y)^{\frac{1}{3}}]^{3}}\end{equation}
and the reciprocal-quartic functional equation

\begin{equation}\label{reci5}q(2x+y)+q(2x-y)= \frac{2q(x)q(y)[q(x)+16q(y)+24\sqrt{q(x)q(y)}]}{[4\sqrt{q(y)}-\sqrt{q(x)}]^{4}}\end{equation}
in non-Archimedean fields.

Very recently, in \cite{vak}, Kumar et al. focused on the stability of  reciprocal-septic functional equation

\begin{multline}\label{reci6}S(2x+y)+S(2x-y)=\\ \frac{2S(x)S(y)[128S(x)+672S(x)^{\frac{2}{7}}S(y)^{\frac{5}{7}}+280S(x)^{\frac{4}{7}}S(y)^{\frac{3}{7}}+14S(x)^{\frac{6}{7}}S(y)^{\frac{1}{7}}]}{[4S(y)^{\frac{2}{7}}-S(x)^{\frac{2}{7}}]^{7}},\end{multline}

and the reciprocal-octic functional equation
\begin{multline}\label{reci7}
\mathcal{O}(2x+y)+\mathcal{O}(2x-y)=\\
\frac{2\mathcal{O}(x)\mathcal{O}(y)[\mathcal{O}(x)+1792\mathcal{O}(x)^{\frac{1}{4}}\mathcal{O}(y)^{\frac{3}{4}}+1120\mathcal{O}(x)^{\frac{1}{2}}\mathcal{O}(y)^{\frac{1}{2}}
+112\mathcal{O}(x)^{\frac{3}{4}}\mathcal{O}(y)^{\frac{1}{4}}+256\mathcal{O}(y)]}
{[4\mathcal{O}(y)^{\frac{1}{4}}-\mathcal{O}(x)^{\frac{1}{4}}]^{8}}.
\end{multline}
The other results pertaining to the stability of different reciprocal-type functional equations can be found in \cite{bo,ra,ra1,ra2, eb,rab,ku1,pa,sadani}.

In this paper, we introduce the following  new reciprocal functional equation  
\begin{equation}\label{eq0}f(2x+y)+f(2x-y)=\frac{2f(x)f(y)\displaystyle{\sum_{\substack{k=0\\ \text{$k$ is even}}}^{ l}2^{l-k}\binom{l}{k}f(x)^{\frac{k}{l}}f(y)^{\frac{l-k}{l}}}}{\left(4f(y)^{\frac{2}{l}}-f(x)^{\frac{2}{l}}\right)^l}\end{equation}

then, we investigate the general solution and its generalized Hyers-Ulam stability  in the non-zero real numbers and in non-Archimedean fields.

\section{General solution of (\ref{eq0})}

In this section, the solution of the functional equation (\ref{eq0}) is given.  Throughout this section, we assume that $\mathbb{R}^*$ is the set of non-zero real numbers.

\begin{theorem}\label{theo1} Let $f: \mathbb{R}^* \to \mathbb{R}$ be a continuous real-valued function of a two real variables satisfying the reciprocal type functional equation (\ref{eq0})  and assume that $f(x)\neq0$ and $4f(y)^{\frac{2}{l}}-f(x)^{\frac{2}{l}}\neq0$ for all $x,y\in \mathbb{R}^*$. Then $f$ is of the form 
$$f(x)=\frac{c}{x^l},$$
for all $x\in \mathbb{R}^*$, where $c\neq0$.
\end{theorem}

 \begin{proof}Let $f : \mathbb{R}^* \to\mathbb{R} $ satisfies the functional equation (\ref{eq0}). Replacing $(x,y)$ by $(x,x)$ in (\ref{eq0}), we obtain
 $$f(3x)+f(x)=\frac{2f(x)\displaystyle{\sum_{\substack{k=0\\ \text{$k$ is even}}}^{ l}2^{l-k}\binom{l}{k}}}{3^l}$$
 By induction, it is  straightforward to show that 
 \begin{equation}\label{eqbinom}\sum_{\substack{k=0\\ \text{$k$ is even}}}^{ l}2^{l-k}\binom{l}{k}=\frac{3^{l}+1}{2},l\geq0.\end{equation}
 Then, 
 $$f(3x)=\frac{1}{3^{l}}f(x)$$
 for all $x\in \mathbb{R}^*$.
 Therefore $f$ is a reciprocal function  of the form $\frac{c}{x^l}$. This completes the proof.
 \end{proof}

\section{Stability of (\ref{eq0}) in non-zero Real Numbers }

%\begin{definition} A mapping $f:\mathbb{R}^*\to\mathbb{R}^*$ is called  $n$-power reciprocal, if the  reciprocal functional equation (\ref{eq0}) holds for all $x,y\in\mathbb{R}^*$.
%\end{definition}

%The proof of the following Lemma is obvious.
%\begin{lemma}\label{lem} We have
%$$\sum_{\substack{k=0\\ \text{$k$ is even}}}^{ l}2^{l-k}\binom{l}{k}=\frac{3^{l}+1}{2},l\geq0.$$
%\end{lemma}
%\begin{proof} The proof is immediate.
%\end{proof}
 For future use and convenience, we introduce the following difference operator $\Lambda:\mathbb{R}^{*}\to \mathbb{R}$ as 
$$\Lambda(x,y)=f(2x+y)+f(2x-y)-\frac{2f(x)f(y)\displaystyle{\sum_{\substack{k=0\\ \text{$k$ is even}}}^{l}2^{l-k}\binom{l}{k}f(x)^{\frac{k}{l}}f(y)^{\frac{l-k}{l}}}}{\left(4f(y)^{\frac{2}{l}}-f(x)^{\frac{2}{l}}\right)^l},$$
with $l\geq1$ and  $x,y\in\mathbb{R}^{*}.$
\begin{theorem}\label{theo1}Let $Q:\mathbb{R}^*\times\mathbb{R}^*\to\mathbb{R}$ a function satisfying 
\begin{equation}\label{ineq0}
\sum_{s=0}^{\infty}\frac{1}{3^{ls}}Q\left(\frac{x}{3^{s+1}},\frac{y}{3^{s+1}}\right)<\infty
\end{equation}
for all $x,y\in\mathbb{R}^*$. If a function $ f : \mathbb{R}^*\to\mathbb{R}$ satisfies functional inequality
\begin{equation}\label{eq1-1}
\left|\Lambda(x,y)\right|\leq Q(x,y)
\end{equation}
for all $x,y \in\mathbb{R}^*$, then there exists a unique  reciprocal function $g:\mathbb{R}^*\to\mathbb{R}$ which satisfies (\ref{eq0}) and the inequality
\begin{equation}\label{ineq1} |f(x)-g(x)|\leq\sum_{s=0}^{\infty}\frac{1}{3^{ls}}Q\left(\frac{x}{3^{s+1}},\frac{x}{3^{s+1}}\right)\end{equation}
for all $x\in\mathbb{R}^*$.
\end{theorem}
\begin{proof}We substitute $(x,y)$ by $(x,x)$ in (\ref{eq1-1}) and using  (\ref{eqbinom}) we get
\begin{equation}\label{eq2}
\left|f(3x)-\frac{f(x)}{3^l}\right|\leq Q(x,x)
\end{equation}
for all $x\in\mathbb{R}^*$. Substituting $x$ by $\frac{x}{3}$ in (\ref{eq2}), we obtain
\begin{equation}\label{eq3}
\left|f(x)-\frac{1}{3^l}f\left(\frac{x}{3}\right)\right|\leq Q\left(\frac{x}{3},\frac{x}{3}\right)
\end{equation}
for all $x\in\mathbb{R}^*$. Now, by replacing $x$ into $\frac{x}{3}$ in (\ref{eq3}), dividing by $3^l$ and then adding the ensuing inequality with (\ref{eq3}), we get
\begin{equation}\label{eq4}
\left|f(x)-\frac{1}{3^{2l}}f\left(\frac{x}{3^2}\right)\right|\leq Q\left(\frac{x}{3},\frac{x}{3}\right)+\frac{1}{3^l}Q\left(\frac{x}{3^2},\frac{x}{3^2}\right)
\end{equation}
for all $x\in\mathbb{R}^*$. In a similar way, proceeding further and applying induction procedure on a positive integer $m$, we get
\begin{equation}\label{eq5}
\left|f(x)-\frac{1}{3^{ml}}f\left(\frac{x}{3^m}\right)\right|\leq \sum_{s=0}^{m-1}\frac{1}{3^{ls}}Q\left(\frac{x}{3^{s+1}},\frac{x}{3^{s+1}}\right)
\end{equation}
for all $x\in\mathbb{R}^*$. 

%%%%%%%%%%%%%%%%%%%%%%%%%%%%%%%%%%%%%%%%%%%%%%%%%%%%%%%%%%%%%%%%
%%%%%%%%%%%%%%%%%%%%%%%%%%%%%%%%%%%%%%%%%%%%%%%%%%%%%%%%%%%%%%%%

%For any positive integer $j$ and $x\in \mathbb{R}^{*}$, we have

%\begin{align*}
 %%\left|\frac{1}{3^{l(j+1)}}f\left(\frac{x}{3^{j+1}}\right)-\frac{1}{3^{l(j+1)}}f\left(\frac{x}{3^{j}}\right)\right|&=\frac{1}{3^{l(j+1)}}\left|f\left(\frac{x}{3^{j}}\right)-\frac{1}{3^{l}}f\left(\frac{x}{3^{j+1}}\right)\right|\\
%&\leq\frac{1}{3^{lj}}Q\left(\frac{x}{3^{j+1}},\ \frac{x}{3^{j+1}}\right) %.\end{align*}
%%%%%%%%%%%%%%%%%%%%%%%%%%%%%%%%%%%%%%%%%%%%%%%%%%%%%
%%%%%%%%%%%%%%%%%%%%%%%%%%%%%%%%%%%%%%%%%%%%%%%%%%%%%%
Next, for any integers $m', m$ with $m'>m>0$, we obtain 

\begin{multline}\label{eq06}
\left|\frac{1}{3^{lm'}}f\left(\frac{x}{3^{m'}}\right)-\frac{1}{3^{lm}}f\left(\frac{x}{3^{m}}\right)\right|
=\left|\frac{1}{3^{lm'}}f\left(\frac{x}{3^{m'}}\right)-\frac{1}{3^{l(m'-1)}}f\left(\frac{x}{3^{l(m'-1)}}\right)+\cdots\right.\\ \quad
{}\left.\left. +\frac{1}{3^{l(m'-1)}}f\left(\frac{x}{3^{l(m'-1)}}\right)-\cdots
+\frac{1}{3^{l(m+1)}}f\left(\frac{x}{3^{m+1}}\right)-\frac{1}{3^{lm}}f\left(\frac{x}{3^{m}}\right)\right|\right.\\
\leq\frac{1}{3^{l(m'-1)}}Q\left(\frac{x}{3^{m'}},\frac{x}{3^{m'}}\right) 
+\cdots+\frac{1}{3^{lm}}Q\left(\frac{x}{3^{m+1}}, \frac{x}{3^{m+1}}\right)\\
\leq\sum_{j=m}^{m'-1}\frac{1}{3^{lj}}Q\left(\frac{x}{3^{j+1}}, \frac{x}{3^{j+1}}\right)
\end{multline}  

for all $x\in \mathbb{R}^{*}$. Taking $ m'\rightarrow\infty$ in (\ref{eq06}) and applying (\ref{ineq0}), we observe that the sequence $ \{ \frac{1}{3^{lm}}f(\frac{x}{3^{m}})\}$ is a Cauchy sequence for each $x\in \mathbb{R}^{*}$. Since $\mathbb{R}$ is complete, we can define $g : \mathbb{R}^{*}\rightarrow \mathbb{R}$ by $g(x)= \lim_{m\rightarrow\infty}\frac{1}{3^{lm}}f\left(\frac{x}{3^{m}}\right)$. To prove that $g$ satisfies (\ref{eq0}), substuting $(x,y)$ by $(3^{-m}x, 3^{-m}y)$ in (\ref{eq1-1}) and dividing by $3^{lm}$, we arrive

\begin{equation}\label{eq7}|3^{-lm}\Lambda(3^{-m}x,\ 3^{-m}y)|\leq 3^{-lm}Q(3^{-m}x,\ 3^{-m}y)\end{equation} 
for all $x, y\in \mathbb{R}^{*}$ and for all positive integer $m$. Letting $ m\rightarrow\infty$ in (\ref{eq7}) and using (\ref{ineq0}), we see that $g$ satisfies (\ref{eq0}) for all $x, y\in \mathbb{R}^{*}$. Again, by taking $ m\rightarrow\infty$ in (\ref{eq5}), we achieve (\ref{ineq1}). Now, it remains to show that $g$ is unique. Let $g' : \mathbb{R}^{*}\rightarrow \mathbb{R}$ be another  reciprocal mapping which satisfies (\ref{eq0}) and the inequality (\ref{ineq1}).
 Obviously, we have $g'(3^{-m}x)=3^{lm}g'(x), g(3^{-m}x)=3^{lm}g(x)$ and using (\ref{ineq1}), we get
\begin{align}\label{eq8}
|g'(x)-g(x)|&=3^{-lm}\left|g'(3^{-m}x)-g(3^{-m}x)\right|\\ \notag
&\leq 3^{-lm}\left(\left|g'(3^{-m}x)-f(3^{-m}x)\right|+\left|f(3^{-m}x)-g(3^{-m}x)\right|\right)\\\notag
&\leq 2\sum_{j=0}^{\infty}\frac{1}{3^{l(m+j)}}Q\left(\frac{x}{3^{m+j+1}},\ \frac{x}{3^{m+j+1}}\right)\\ \notag
& \leq 2\sum_{j=m}^{\infty}\frac{1}{3^{lj}}Q\left(\frac{x}{3^{j+1}}, \frac{x}{3^{j+1}}\right)
\end{align}   

for all $x\in \mathbb{R}^{*}$. Letting $ m\rightarrow\infty$ in (\ref{eq8}), we find that $g$ is unique. The proof of Theorem \ref{theo1} is complete.
\end{proof}
The following corollaries are  immediate consequences of Theorem \ref{theo1}.
\begin{corollary}
Let $f:\mathbb{R}^{*}\rightarrow \mathbb{R}$  be a mapping for which there exists  $\epsilon>0$ such that 
$$
|\Lambda(x,y)|\leq\epsilon
$$
 holds for all $x, y\in \mathbb{R}^{*}$.  Then,
$$
g(x)=\lim_{m\rightarrow\infty}\frac{1}{3^{lm}}f\left(\frac{x}{3^{m}}\right)
$$
 for all $x\in \mathbb{R}^{*}, m\in \mathbb{N}$  and $g : \mathbb{R}^{*}\rightarrow \mathbb{R}$  is the unique mapping satisfying the reciprocal functional equation (\ref{eq0})   such that
$$
|f(x)-g(x)|\leq\frac{3^l}{3^l-1}\epsilon
$$
 for every $x\in \mathbb{R}^{*}.$
\end{corollary}

\begin{proof} By taking $ Q(x, y)=\epsilon$ in Theorem \ref{theo1}  we arrive at the desired result.
\end{proof}
\begin{corollary}  Let $\epsilon>0$  and $\alpha\neq-l$  be real numbers, and $f: \mathbb{R}^{*}\rightarrow \mathbb{R}$  be a mapping satisfying the functional inequality

$$
|\Lambda(x,y)|\leq\epsilon(|x|^{\alpha}+|y|^{\alpha})
$$
 for all $x, y\in \mathbb{R}^{*}.$  Then, there exists a unique  reciprocal mapping $g:\mathbb{R}^{*}\rightarrow \mathbb{R}$  satisfying the functional equation (\ref{eq0})   and
$$
|f(x)-g(x)|\leq\frac{2.3^{l}\epsilon}{3^{\alpha+l}-1}|x|^{\alpha}
$$
 for all $x\in \mathbb{R}^{*}$.
\end{corollary}
\begin{proof} By letting $Q(x,y)=\epsilon(|x|^{\alpha}+|y|^{\alpha})$ for all $x, y\in \mathbb{R}^{*}$  in Theorem \ref{theo1} we get the desired result.
\end{proof}

\begin{corollary}  Let $\epsilon>0$  and $\alpha\neq-l$ {\it be real numbers, and} $f: \mathbb{R}^{*}\rightarrow \mathbb{R}$  be a mapping satisfying
$$
|\Lambda(x,y)|\leq\epsilon(|x|^{\frac{\alpha}{2}}|y|^{\frac{\alpha}{2}}+|x|^{\alpha}+|y|^{\alpha})
$$
 for all $x, y\in \mathbb{R}^{*}.$  Then, there exists a unique  reciprocal mapping $g:\mathbb{R}^{*}\rightarrow \mathbb{R}$  satisfying the functional equation (\ref{eq0})   and
$$
|f(x)-g(x)|\leq\frac{\epsilon3^{l+1}}{3^{\alpha+l}-1}|x|^{\alpha}
$$
 for all $x\in \mathbb{R}^{*}$.\end{corollary}

\begin{proof} By taking $Q(x, y)=\epsilon(|x|^{\frac{\alpha}{2}}|y|^{\frac{\alpha}{2}}+|x|^{\alpha}+|y|^{\alpha})$ for all $x, y\in \mathbb{R}^{*}$ in Theorem \ref{theo1} we get the desired result.
\end{proof}
\begin{corollary}  Let $f : \mathbb{R}^{*}\rightarrow \mathbb{R}$ {\it be a mapping and there exist} $p, q$  with $ p+q\neq-l$. {\it If there exists} $\epsilon\geq 0$ {\it such that}
$$
|\Lambda(x,y)|\leq\epsilon|x|^{p}|y|^{q}
$$
 for all $x, y\in \mathbb{R}^{*}$,  then there exists a unique  reciprocal mapping $g : \mathbb{R}^{*}\rightarrow \mathbb{R}$  satisfying the functional equation (\ref{eq0})  and
$$
|f(x)-g(x)|\leq\frac{3^{l}\epsilon}{3^{p+q+l}-1}|x|^{p+q}$$
 for all $x\in \mathbb{R}^{*}$.
\end{corollary}

\begin{proof} Letting $Q(x,y)=\epsilon|x|^{p}|y|^{q}$ for all $x, y\in \mathbb{R}^{*}$ in Theorem \ref{theo1}, we obtain the required result.
\end{proof}
\section{Stability of (\ref{eq0}) in non-Archimedean field} 
Throughout this section, $\mathbb{A}$ and $\mathbb{B}$ will denote a non-Archimedean
field and a complete non-Archimedean field, respectively. Moreover, for a non-Archimdean field $\mathbb{A}$, we put $\mathbb{A}^*=\mathbb{A}-\left\{0\right\}$ . We also assume that the reader is familiar with the basic properties of  non-Archimedean fields.

\begin{theorem}\label{theo2}  Let $ G:\mathbb{A}^{*}\times \mathbb{A}^{*}\rightarrow[0,\ \infty[$  be a mapping such that
\begin{equation}\label{c0} \lim_{m\rightarrow\infty}\left|\frac{1}{3^l}\right|^{m}G\left(\frac{x}{3^{m+1}},\ \frac{y}{3^{m+1}}\right)=0   
\end{equation}
{\it for all} $x, y\in \mathbb{A}^{*}$.  Suppose that $g:\mathbb{A}^{*}\rightarrow \mathbb{B}$ is a mapping satisfying the inequality
\begin{equation}\label{c1}|\Lambda(x,y)|\leq G(x, y)\end{equation}
 for all $x, y\in \mathbb{A}^{*}$.  Then, there exists a unique  reciprocal mapping $g:\mathbb{A}^{*}\rightarrow \mathbb{B}$  such that  such that

\begin{equation}\label{c2}| f(x)-g(x)|\leq\max \left\{ \left|\frac{1}{3^l}\right|^{k+1}G\left(\frac{x}{3^{k+1}}, \frac{x}{3^{k+1}}\right) : k\in \mathbb{N}\cup\{0\}\right\}
 \end{equation}
  for all $x\in \mathbb{A}^{*}$.
\end{theorem}

\begin{proof} Replacing $(x, y)$ by $(x,\ x)$ in (\ref{c1}), one finds

\begin{equation}\label{c3}\left|f(x)- \frac{1}{3^l}f\left(\frac{x}{3}\right)\right|\leq|3^{l} |G\left(\frac{x}{3}, \frac{x}{3}\right)\end{equation}
for all $x\in \mathbb{A}^{*}$. Now, considering $x$ as $ \frac{x}{3^{m}}$ in (\ref{c3}) and multiplying by $\left|\frac{1}{3^l}\right|^{m}$, we get

\begin{equation}\label{c4}
\left| \frac{1}{3^{lm}}f\left(\frac{x}{3^{m}}\right)-\frac{1}{3^{l(m+1)}}f\left(\frac{x}{3^{(m+1)}}\right)\right|
 \leq\left|\frac{1}{3^l}\right|^{m}G\left(\frac{x}{3^{m+1}}, \frac{x}{3^{m+1}}\right)
\end{equation}

for all $x\in \mathbb{A}^{*}$. It is easy to obtain from the inequalities (\ref{c0}) and (\ref{c4}) that the sequence $\left\{ \frac{1}{3^{lm}}f\left(\frac{x}{3^{lm}}\right)\right\}$ is Cauchy. Since $\mathbb{B}$ is complete, we can say that this sequence converges to a mapping $g:\mathbb{A}^{*}\rightarrow \mathbb{B}$ defined by
\begin{equation}\label{c5}
g(x)= \lim_{m\rightarrow\infty}\frac{1}{3^{lm}}f\left(\frac{x}{3^{m}}\right) \end{equation}
Besides, for each $x\in \mathbb{A}^{*}$ and nonnegative integers $m$, we obtain 
\begin{align}\label{c6}
\left|\frac{1}{3^{lm}}f\left(\frac{x}{3^{m}}\right)-g(x)\right|&=\left|\sum_{k=0}^{m-1}\left[\frac{1}{3^{l(k+1)}}f\left(\frac{x}{3^{(k+1)}}\right)-\frac{1}{3^{lm}}f\left(\frac{x}{3^{m}}\right)\right]\right|\notag
\\ 
&\leq\max\left\{\left|\frac{1}{3^{l(k+1)}}\right|f\left(\frac{x}{3^{(k+1)}}\right)-\frac{1}{3^{lm}}f\left(\frac{x}{3^{m}}\right)\right|:0\leq  k<m\}\\ \notag
& \leq\max\left\{\left|\frac{1}{3^l}\right|^{m}G\left(\frac{x}{3^{m+1}},\frac{x}{3^{m+1}}\right):0\leq k<m\right\}.
\end{align}

Taking $ m\rightarrow\infty$ in the inequality (\ref{c6}) and using (\ref{c5}), we observe that the inequality (\ref{c2}) is valid. By the application of inequalities (\ref{c0}), (\ref{c1}) and (\ref{c5}), for all $x, y\in \mathbb{A}^{*},$ we arrive at
\begin{align*}
|\Lambda(x,y)|&=\lim_{m\rightarrow\infty}\left|\frac{1}{3^l}\right|^{m}\left|\Lambda\left(\frac{x}{3^{m}},\ \frac{y}{3^{m}}\right)\right|\\
&\leq\lim_{m\rightarrow\infty}\left|\frac{1}{3^l}\right|^{m}G\left(\frac{x}{3^{m}},\ \frac{y}{3^{m}}\right)\\
&=0.
\end{align*}
Hence, the mapping $g$ satisfies (\ref{c0}) and so it is  reciprocal mapping. In order to confirm the uniqueness of $g$, let us presume $g'$ : $\mathbb{A}^{*}\rightarrow \mathbb{B}$ be another reciprocal mapping verifying (\ref{c2}). Then

\begin{multline*}|g(x)-g'(x)|=\lim_{n\rightarrow\infty}\left|\frac{1}{3^l}\right|^{n}\left|g\left(\frac{x}{3^{n}}\right)-g'\left(\frac{x}{3^{n}}\right)\right|\\
\leq\lim_{n\rightarrow\infty}\left|\frac{1}{3^l}\right|^{n}\max\left\{\left|g\left(\frac{x}{3^{n}}\right)-f\left(\frac{x}{3^{n}}\right)\right|, \left|f\left(\frac{x}{3^{n}}\right)-g'\left(\frac{x}{3^{n}}\right)\right|\right\}\\
\leq\lim_{n\rightarrow\infty}\lim_{m\rightarrow\infty}\max\left\{\max\left\{
\left| \frac{1}{3^l}\right|^{k+n}G\left(\frac{x}{3^{k+n+1}}, \frac{x}{3^{k+n+1}}\right)\right.\right.\\ \quad
{}\left.\left.\!\! :n\leq k\leq m+n\right\}\right\} \\ 
=0\\
\end{multline*}
for all $x\in \mathbb{A}^{*}$. This implies that $g$ is unique, which completes the proof.
\end{proof}
 As a direct consequence of Theorem \ref{theo2}, we have the  following corollaries.

\begin{corollary}  Let $\mu>0$  be a constant. If $f : \mathbb{A}^{*}\rightarrow \mathbb{B}$  satisfies
 $$|\Lambda(x,y)|\leq\mu$$ 
 for all $x, y\in \mathbb{A}^{*}$,  then there exists a unique  reciprocal mapping $g : \mathbb{A}^{*}\rightarrow \mathbb{B}$  satisfying (\ref{eq0})  and $$|f(x)-g(x)|\leq\mu$$  for all $u\in \mathbb{A}^{*}$.\end{corollary}

\begin{proof} Taking $ G(x, y)=\mu$ in Theorem \ref{theo2}, we get  the required result.\end{proof}
\begin{corollary}  Let $\mu\geq 0$  and $a\neq-l$,  be fixed constants. If $f : \mathbb{A}^{*}\rightarrow \mathbb{B}$  satisfies $$|\Lambda(x, y)|\leq\mu(|x|^{a}+|y|^{a})$$ {\it for all} $x, y\in \mathbb{A}^{*}$,  then there exists a unique  reciprocal  mapping $g:\mathbb{A}^{*}\rightarrow \mathbb{B}$  satisfying (\ref{eq0})  and
$$|f(x)-g(x)|\leq\left\{\begin{array}{ll}
\frac{|2|\mu}{|3|^{a}}|u|^{a}, & a>-l,\\
|2|\mu|3|^{l}|u|^{a}, & a<-l,
\end{array}\right.$$

 for all $x\in \mathbb{A}^{*}$. 
 \end{corollary}

\begin{proof} Considering $G(x, y)=\mu(|x|^{a}+|y|^{a})$ in Theorem \ref{theo2}, we obtain the desired result.\end{proof}
\begin{corollary}  Let $f : \mathbb{A}^{*}\rightarrow \mathbb{B}$  be a mapping and let there exist real numbers $p, q, a=p+q\neq-l$ {\it and} $\mu\geq 0$  such that $$|\Lambda(x,y)|\leq\mu|u|^{p}|v|^{q}$$ {\it for all} $x, y\in \mathbb{A}^{*}$. Then, there exists a unique  reciprocal mapping $g:\mathbb{A}^{*}\rightarrow \mathbb{B}$  satisfying (\ref{eq0}) and

$$|f(x)-g(x)|\leq\left\{\begin{array}{ll}
\frac{\mu}{|3|^{a}}|u|^{a}, & a>-l,\\
\mu|3|^{l}||u|^{a}, & a<-l,
\end{array}\right.$$
 for all $x\in \mathbb{A}^{*}$.\end{corollary}

\begin{proof} Letting $G(x, y)=\mu|x|^{p}|y|^{q}$, for all $x, y\in \mathbb{A}^{*}$ in Theorem \ref{theo2}, we acquire the requisite result.
\end{proof}
%%%%%%%%%%%%%%%%%%%%%%%%%%%%%%%%%%%%%%%%%%%%EXACT%%%%%%%%%%%%%%%%%%%%%%%%%%%%%%%%%
\begin{corollary}  Let $\mu\geq 0$ {\it and} $a\neq-l$  be real numbers, and $f : \mathbb{A}^{*}\rightarrow \mathbb{B}$  be a mapping satisfying the functional inequality
$$
|\Lambda(x,y)|\leq\mu(|x|^{\frac{a}{2}}|y|^{\frac{a}{2}}+|x|^{a}+|y|^{a})
$$
 for all $x, y\in \mathbb{A}^{*}$ {\it Then, there exists a unique  reciprocal mapping} $g:\mathbb{A}^{*}\rightarrow \mathbb{B}$ satisfying (\ref{eq0})  and

$$|f(x)-g(x)|\leq\left\{\begin{array}{ll}
\frac{|3|\mu}{|3|^{a}}|u|^{a}, & a>-l,\\
|3|\mu|3|^{l}|u|^{a}, & a<-l,
\end{array}\right.$$

{\it for all} $x\in \mathbb{A}^{*}$. \end{corollary}

\begin{proof} Letting $G(x,y)=\mu(|x|^{\frac{a}{2}}|y|^{\frac{a}{2}}+|x|^{a}+|y|^{a})$ in Theorem \ref{theo2}, the result follows directly.\end{proof}

\begin{proof} Considering $\mu(x,\ y)=\epsilon(|x|^{a}+|y|^{a})$ in Theorem \ref{theo2} the desired result follows directly.
\end{proof}
\begin{corollary} Let $\alpha : [0, +\infty|\rightarrow]0, +\infty[$  be a function satisfying
$$
\alpha(|3|^{-1}t)\leq\alpha(|3|^{-1})\alpha(t)\ (t\geq 0), \left|\frac{1}{3^l}\right|\alpha(|3|^{-1})<1.$$
 Let $\delta>0$  and $f : \mathbb{A}^*\rightarrow \mathbb{B}$  be a mapping
 satisfying
   
\begin{equation}|\Lambda(x,y)| \leq\delta(\alpha(|x|)+\alpha(|y|))\end{equation} 
with  $x, y\in \mathbb{A}^*$.
 Then there exists a unique  reciprocal mapping $h:\mathbb{A}^*\rightarrow \mathbb{B}$  such that
$$
| f(x)-h(x)|\leq2\delta\alpha\left(\left|\frac{x}{3}\right|\right) \ \  (x\in \mathbb{A}^*).
$$
\end{corollary}
\begin{proof}Defining $G :  \mathbb{A}^*\times \mathbb{A}^*\rightarrow]0,+\infty[$ by $G(x,\ y) =\delta(\alpha(|x|)+\alpha(|y|))$. On the one hand, we have

\begin{align*}
 \lim_{m\to\infty}\left|\frac{1}{3^l}\right|^m  G\left(\frac{x}{3^{m+1}},\frac{y}{3^{m+1}}\right) 
 &=\lim_{m\to\infty}\left|\frac{1}{3^l}\right|^m \delta\left(\alpha\left(\left|\frac{x}{3^{m+1}}\right|\right)+\alpha\left(\left|\frac{y}{3^{m+1}}\right|\right)\right)\\
 &\leq\lim_{m\to\infty}\left|\frac{1}{3^l}\right|^m\alpha(|3|^{-1})^{m+1}G(x, y)\\
&=\lim_{m\to\infty}\left(\left|\frac{1}{3^l}\right|\alpha(|3|^{-1})\right)^{m}\alpha(|3|^{-1})G(x, y)\\
&=\lim_{m\to\infty}\left(\left|\frac{1}{3^l}\right|\alpha(|3|^{-1})\right)^{m+1}\left|\frac{1}{3^l}\right|^{-1}G(x, y)\\&=0.
\end{align*}
for all $x, y\in \mathbb{A}^*$. On the other hand,
$$\lim_{m\rightarrow\infty}\max\left\{\left|\frac{1}{3^l}\right|^k  G\left(\frac{x}{3^{k+1}},\frac{x}{3^{k+1}}\right) 
 ;0\leq k<m\right\}=G\left(\frac{x}{3},\frac{x}{3}\right)=2\delta\alpha\left(\left|\frac{x}{3}\right|\right),$$
for all $x\in\mathbb{A}^*$, and
\begin{align*} 
 \lim_{j\rightarrow\infty}\lim_{m\rightarrow\infty}\max\{\left|\frac{1}{3^l}\right|^k  G\left(\frac{x}{3^{k+1}},\frac{x}{3^{k+1}}\right),j\leq k<m+j\}&=\lim_{j\to\infty} \left|\frac{1}{3^l}\right|^j  G\left(\frac{x}{3^{j+1}},\frac{x}{3^{j+1}}\right) \\
& =0.\end{align*}
Applying Theorem \ref{theo1}, we get desired result.
\end{proof}
\section{Conclusion}
\begin{enumerate}
\item The advantage of this paper is that we do not need to prove the stability of  quadratic (\ref{reci2}),  quartic (\ref{reci5}), septic (\ref{reci6}) and octic (\ref{reci7}) reciprocal functional equations separately. Instead we can apply our main theorem to prove the stability of those functional equations simultaneously.
\item Similarly, we can give the generalized form of the functional equation (\ref{reci3}) and (\ref{reci4}) as follows
\begin{equation}\label{eqnp}f(2x+y)+f(x+2y)=\frac{f(x)f(y)\sum_{k=0}^{n}(2^{n-k}+2^k)\binom{n}{k}f(x)^{\frac{k}{n}}f(y)^{\frac{n-k}{n}}}{\left(2f(x)^{\frac{2}{n}}+5f(x)^{\frac{1}{n}}f(y)^{\frac{1}{n}}+2f(y)^{\frac{2}{n}}\right)^n}, .\end{equation}
with $n\geq1$. The  reciprocal function $f(x)=\frac{c}{x^n}$ is solution of (\ref{eqnp}). Then the generalized Ulam-Hyers stability can be obtained in the same way as above study or as  \cite{ra14} in various spaces.\end{enumerate}


\end{document}